\documentclass[12pt]{article}
\usepackage{amsmath,amsfonts,amssymb,amsthm,epsfig,epstopdf,titling,url,array}
\usepackage{eqnarray}
\usepackage{enumerate}
\usepackage[a4paper]{geometry}
\usepackage{amscd}
\usepackage{centernot}
\usepackage{cleveref}
\usepackage{url}
\usepackage{authblk}
\usepackage{setspace}
\usepackage{tikz-cd}
\newtheorem{thm}{ Theorem}[section]
\newtheorem{lem}[thm]{Lemma}
\newtheorem{prop}[thm]{Proposition}
\newtheorem{cor}[thm]{Corollary}
\newtheorem{defn}[thm]{Definition}
\newtheorem{ex}[thm]{Example}

\newtheorem* {note*}{Note}
\newtheorem* {con*}{Conjecture}

\newcommand{\co} {\mathbb{C}}

\newcommand{\iy} {\infty}
\newcommand{\N} {\mathbb{N}}

\newcommand{\D} {\mathbb{D}}
\newcommand{\bo}{\mathbb{D}\backslash \left\{0\right\}}

\newcommand{\U} {\mathbb{U}}
\newcommand{\uc} {\mathbb{U}^c}
\newcommand{\Z} {\mathbb{Z}}

\newcommand{\al}{\alpha}
\newcommand{\lm}{\lambda}

\newcommand{\x}{\mathcal X}

\newcommand{\B}{\mathcal B}

\newcommand{\bx}{\mathcal B(\mathcal X)}

\newcommand{\norm}[1]{\left\Vert#1\right\Vert}

\newcommand{\set}[1]{\left\{#1\right\}}

\newcommand{\seq}[1]{\left<#1\right>}
\newcommand{\brc}[1]{\left(#1\right)}

\newcommand{\LNp}{\ell^{p}(\mathbb N)}
\newcommand{\LZp}{\ell^{p}(\mathbb Z)}
\newcommand{\m}{\mathcal {M}}
\newcommand{\n}{\mathcal {N}}

\usepackage{geometry}
\title{\bf \Large On Subspace-diskcyclicity}
\bf
\author[1]{\bf\footnotesize Nareen Bamerni \thanks{nareen\_bamerni@yahoo.com}}
\author[2]{\bf Adem K{\i}l{\i}\c{c}man \thanks{akilicman@yahoo.com}}
\affil[1,2]{\bf Department of Mathematics, University Putra Malaysia,
43400 UPM, Serdang, Selangor, Malaysia}

\begin{document}
\date{}
\maketitle

\begin{abstract}
In this paper, we define and study subspace-diskcyclic operators. We show that subspace-diskcyclicity does not imply to diskcyclicity. We establish a subspace-diskcyclic criterion and use it to find a subspace-diskcyclic operator that is not subspace-hypercyclic for any subspaces. 
Also, we show that the inverse of invertible subspace-diskcyclic operators do not need to be subspace-diskcyclic for any subspaces. Finally, we prove that every finite-dimensional separable Hilbert space over the complex field supports a subspace-diskcyclic operator.
\end{abstract}

{\bf keywords}:diskcyclic operators, Dynamics of linear operators in Banach spaces.

\section{introduction}
\label{1}
A bounded linear operator $T$ on a separable Banach space $\x$ is hypercyclic if there is a vector $x\in \x$ such that $Orb(T,x)=\set{T^nx:n\ge 0}$ is dense in $\x$, such a vector $x$ is called  hypercyclic for $T$. The first example of a hypercyclic operator on a Banach space was constructed by Rolewicz in 1969 \cite{rolewicz1969orbits}. He showed that if $B$ is the backward shift on $\LNp$ then $\lm B$ is hypercyclic if and only if $|\lm|> 1$. \\

The studying of the scaled orbit and  disk orbit are motivated by the Rolewicz example \cite{rolewicz1969orbits}. In 1974, Hilden and Wallen \cite{hilden1974some} defined the supercyclicity concept. An operator $T$ is called supercyclic if there is a vector $x$ such that the cone generated by $Orb(T,x)$ is dense in $\x$. The notion of a diskcyclic operator was introduced by Zeana \cite{zeana2002cyclic}. An operator $T$ is called diskcyclic if there is a vector $x \in \x$ such that the disk orbit $\D Orb(T,x)=\set{\al T^nx: \al\in \co, |\al|\le 1, n\in \N}$ is dense in $\x$, such a vector $x$ is called  diskcyclic for $T$. For more information about diskcyclic operators, the reader may refer to \cite{bamerni2015review} \cite{Bamerni2015414} \cite{zeana2002cyclic}. \\

In 2011,  Madore and Mart\'{i}nez-Avenda\~{n}o \cite{madore2011subspace} considered the density of the orbit in a non-trivial subspace instead of the whole space, this phenomenon is called   the subspace-hypercyclicity. An operator is called $\m$-hypercyclic or subspace-hypercyclic for a subspace $\m$ of $\x$ if there exists a vector such that the intersection of its orbit and $\m$ is dense in $\m$. They proved that subspace-hypercyclicity is infinite dimensional phenomenon. For more information on subspace-hypercyclicity, one may refer to \cite{le2011subspace} and \cite{rezaei2013notes}\\

 In 2012 Xian-Feng et al \cite{xian2012subspace} defined the subspace-supercyclic operator as follows: An operator is called $\m$-supercyclic or subspace-supercyclic for a subspace $\m$ of $\x$  if there exists a vector such that the intersection of the cone generated by its orbit and $\m$ is dense in $\m$.\\

Since both subspace-hypercyclicity and subspace-supercyclicity were studied. It is natural to define and study subspace-diskcyclicity. 
In the second section of this paper, we introduce the concept of subspace-diskcyclicity and subspace-disk transitivity. We show that not every subspace-diskcyclic operator is diskcyclic. We give the relation between all subspace-cyclicity. In particular, we give a set of sufficient conditions for an operator to be subspace-diskcyclic. We use this result to give an example of a subspace-diskcyclic which is not subspace-hypercyclic. Also, we give an example of a supercyclic operator that is not subspace-diskcyclic. Moreover, we give a simple example to show that the inverse of subspace-diskcyclic operators do not need to be subspace-diskcyclic which answers the corresponding question to \cite[Question 1]{xian2012subspace} for subspace-diskcyclicity. As a consequence of this example, we show that subspace-diskcyclicity exists on every finite dimensional Hilbert space which is not true for subspace-hypercyclicity. \\

\section{Main results}
In this paper, all Banach spaces $\x$ are infinite dimensional (unless stated otherwise) and separable over the field $\co$ of complex numbers. 
All subspaces of $\x$ are assumed to be nontrivial linear subspaces and topologically closed, and all relatively open sets are assumed to be nonempty. We will denote the closed unit disk by $\D$ and the open unit disk by $\U$.\\

\begin{defn}\label{100}
Let $T\in\bx$, and let $\m$ be a subspace of $\x$. Then $T$ is called a subspace-diskcyclic operator for $\m$ (or $\m$-diskcyclic, for short)  if there exists a vector $x$ such that $\D Orb(T,x)\cap \m $ is dense in $\m$. Such a vector $x$ is called a subspace-diskcyclic (or $\m$-diskcyclic, for short) vector for $T$.
\end{defn}
%----------------------------------------------------------------------------------------------------------------------------------
Let $\D C(T,\m)$ be the set of all $\m$-diskcyclic vectors for $T$, that is
\begin{center} $\D C(T,\m)=\{x \in \x:\D Orb(T,x)\cap \m$ is dense in $\m\}$. \end{center}
Let $\D C(\m,\x)$ be the set of all $\m$-diskcyclic operators on $\x$, that is
 $$\D C(\m,\x)=\{T \in \bx:\D Orb(T,x)\cap \m\mbox{ is dense in } \m \mbox{ for some } x\in \x \}.$$
%----------------------------------------------------------------------------------------------------------------------------------

By \cite[Theorem 2.1]{m}, every diskcyclic operator is subspace-diskcyclic; on the other hand, the next example shows that the subspace-diskcyclicity does not imply to the diskcyclicity.
\begin{ex}\label{s no d}
Suppose that $T$ is a diskcyclic operator on $\x$, and $x$ is a diskcyclic vector for $T$. Suppose that $\n=\x \oplus \{0\}$, and  $I$ is the identity operator on $\co^2$. Then, the operator $S=T \oplus I \in \B(\x \oplus  \co^2)$ is not diskcyclic on $\x \oplus \x$; otherwise, we get $I$ is diskcyclic operator on $\co^2$ (see \cite[Proposition 2.2]{bamerni2015review}) which contradicts \cite[Proposition 2.1]{bamerni2015review}. However, it is clear that $S$ is $\n$-diskcyclic operator, and $(x,0)$ is $\n$-diskcyclic vector for $S$. 
\end{ex}

From \Cref{s no d} above, it is clear that the \cite[Proposition 2.2]{bamerni2015review} can not be extended to subspace-diskcyclic operators, since $I$ can not be subspace-diskcyclic for any nontrivial subspace.

%----------------------------------------------------------------------------------------------------------------------------------
\begin{defn}\label{def-sbdtr}
Let $T\in \bx$ and $\m$ be a subspace of $\x$. Then $T$ is called subspace-disk transitive  for $\m$ (or $\m$-disk transitive, for short) if for any two  relatively open sets  $U $ and $V $ in $\m$, there exist $n \in \N$ and $\alpha \in \uc$  such that $T^{-n}(\alpha U) \cap V$ contains a relatively open subset $G$ of $\m$.  
\end{defn}
The next lemma gives some equivalent assertions to subspace-disk transitive, which will be the tool to prove several facts in this paper. 
%----------------------------------------------------------------------------------------------------------------------------------
\begin{lem}\label{20}
Let $T \in \bx$ and  $\m$ be a subspace of $\x$. Then the following assertions are equivalent:
\begin{enumerate}
\item $T$ is $\m$-disk transitive,\label{a40}
\item For any two relatively open sets $U$ and $V$ in $\m$, there exist $\alpha \in \uc$ and $n \in \N$ such that $T^{-n}(\alpha U) \cap V$ is nonempty and $T^n(\m)\subset \m.$\label{b40}
\item For any two relatively open sets $U$ and $V$ in $\m$, there exist $\alpha \in \uc$ and $n \in \N$ such that $T^{-n}(\alpha U) \cap V$ is nonempty and open in $\m$.\label{c40}
\end{enumerate}
\begin{proof}
(\ref{a40}) $\Rightarrow$ (\ref{b40}): Let $U$ and $V$ be two open subsets of $\m$. By condition (\ref{a40}), there exist $\alpha \in \uc$, $n\in\N$ and an open set $G$ in $\m$ such that $G \subset T^{-n}(\alpha U) \cap V$. It follows that 
\begin{equation}\label{102}
T^{-n}(\alpha U) \cap V \mbox{ is nonempty }.
\end{equation}
Since $ G \subset T^{-n}(\alpha U)$ it follows that $\frac{1}{\alpha}T^{n}G \subset U \subset \m$. Let $x \in \m$ and $x_0 \in G$. Then there exists  $r \in \N$ such that $(x_0 + rx) \in G$. Then, we get  
$$\frac{1}{\alpha}T^{n}x_0+\frac{1}{\alpha}T^{n}rx = \frac{1}{\alpha}T^{n}(x_0+rx) \in \frac{1}{\alpha}T^{n}G \subset \m.$$
 Since $x_0 \in G$ then $\frac{1}{\alpha}T^{n}x_0\in \frac{1}{\alpha}T^{n}G \subset \m$, it follows that $\frac{r}{\alpha}T^{n}x \in \m$ and so 
\begin{equation}\label{103}
T^{n}x \in \m.
\end{equation}
The proof follows by \Cref{102} and \Cref{103}.\\
(\ref{b40})$\Rightarrow$ (\ref{c40}): Since $T^n|_{\m}\in \B(\m)$, then $T^{-n}(\alpha U) \cap \m$ is open in $\m$ for any open set $U$ of $\m$. Since $V \subset \m$ is open, it follows that $T^{-n}(\alpha U)\cap V$ is an  open set in $\m$.\\
(\ref{c40})$\Rightarrow$ (\ref{a40}) is trivial.
\end{proof}
\end{lem}
%----------------------------------------------------------------------------------------------------------------------------------
The next theorem shows that every subspace-disk transitive operator is subspace-diskcyclic for the same subspace. First, we need the following lemma. \\

We will suppose that $\{B_k:k\in \N\}$ is a countable open basis for the relative topology of a subspace $\m$.
\begin{lem}\label{205}
Let $T$ be an $\m$-diskcyclic operator. Then
$$\displaystyle{\D C(T,\m)=\bigcap_{k\in\N}\big(\bigcup_{\stackrel{\alpha \in \uc}{n\in \N}}T^{-n}(\alpha B_k))}.$$
\begin{proof}
We have $x \in \D C(T,\m)$ if and only if $\{\alpha T^nx: n\in\N, \alpha\in \bo \}\cap \m$ is dense in $\m$ if and only if for each $k>0$, there are $\alpha \in \bo$ and $n \in \N$ such that $\alpha T^nx \in B_k$ if and only if $\displaystyle{ x\in \bigcap_{k\in\N}\big(\bigcup_{\stackrel{\alpha \in \uc}{n\in \N}}T^{-n}(\alpha B_k))}$.
\end{proof}
\end{lem}
%----------------------------------------------------------------------------------------------------------------------------------
\begin{thm}\label{206}
Let $T \in \bx$, and let $\m$ be a subspace of $\x$. Suppose that $T$ is $\m$-disk transitive. Then $\displaystyle{\bigcap_k\big(\bigcup_{\stackrel{\alpha \in \uc}{n\in \N}}T^{-n}(\alpha B_k))}$ is dense in $\m$.
\begin{proof}
Since $T$ is $\m$-transitive, then by \Cref{20}, for each $i,j\in \N$, there exist $n_{i,j}\in \N$ and $\alpha_{i,j} \in \uc$ such that 
$$T^{-n_{i,j}} (\alpha_{i,j}B_i)\cap B_j$$
is nonempty open in $\m$. 
Suppose that
 $$\displaystyle A_i=\bigcup_{j=1}^\iy \left( T^{-n_{i,j}}(\alpha_{i,j}B_i)\cap B_j\right)$$
for all $i\in \N$. Then $A_i$ is nonempty and open in $\m$ since it is a countable union of open sets in $\m$. Furthermore, each $A_i$ is dense in $\m$ since it intersects each $B_j$. By the Baire category theorem, we get
$$\displaystyle \bigcap_{i=1}^\iy A_i=\bigcap_{i=1}^\iy \bigcup_{j=1}^\iy \left( T^{-n_{i,j}}(\alpha_{i,j}B_i)\cap B_j \right)$$
is a dense set in $\m$. Clearly,
$$\bigcap_{i\in\N}\bigcup_{j\in\N}T^{-n_{i,j}}(\alpha_{i,j}B_i)\cap B_j \subset \bigcap_i\bigcup_{\stackrel{\alpha \in \uc}{n\in \N}} T^{-n}(\alpha B_i)\cap \m. $$
It follows that $\bigcap_{i\in\N}\bigcup_{\stackrel{\alpha \in \uc}{n\in \N}} T^{-n}(\alpha B_i)\cap \m $ is desne in $\m$. The proof is completed.
\end{proof}
\end{thm}
%----------------------------------------------------------------------------------------------------------------------------------
\begin{cor}\label{21}
If $T$ is  an $\m$-disk transitive operator, then $T$ is $\m$-diskcyclic.
\end{cor}
\begin{proof}
The proof follows by \Cref{205} and \Cref{206}.
\end{proof}
%----------------------------------------------------------------------------------------------------------------------------------
It is clear from \Cref{100}, that every $\m$-hypercyclic operator is $\m$-diskcyclic which in turn is $\m$-supercyclic. On the other hand, the following two examples show that the reversed directions are not true ingeneral. First we need the following lemma, which extend the diskcyclic criterion to subspace-diskcyclic criterion.
 %----------------------------------------------------------------------------------------------------------------------------------
\begin{lem}[$\m$-Diskcyclic Criterion]\label{sdc}
Let $T \in \bx$ and $\m$ be a subspace of $\x$. Suppose that $\seq{n_k}_{k\in\N}$ is an increasing sequence of positive integers and $D_1, D_2 \in \m$ are two dense sets in $\m$  such that
\begin{enumerate}[(a)]
\item  For every $y\in D_2$, there is a sequence $\seq{x_k}_{k\in\N}$ in $\m$ such that $\norm{x_k} \to 0$ and $T^{n_k}x_k \to y$ as $k\to\iy$,\label{a104}
\item $\left\|T^{n_k}x\right\|\left\|x_k\right\|\to 0$  for all $x \in D_1$ as $k\to\iy$,\label{b104}
\item $T^{n_k}\m\subseteq \m$ for all $k \in \N$.\label{c104}
\end{enumerate}
Then $T$ is said to be satisfied $\m$-diskcyclic criterion, and $T$ is an $\m$-diskcyclic operator.
\end{lem}
\begin{proof}
To show that $T$ is $\m$-diskcyclic operator, we will use the same lines as in the proof of \cite[Theorem 1.14]{bayart2009dynamics}. Let $U_1$ and $U_2$ be two  relatively open sets in $\m$. Then we can find $x \in D_1 \cap U_1$ and $y\in D_2 \cap U_2$ since both $D_1$ and $D_2$ are dense in $\m$. It follows from the condition \ref{b104} that there exists a sequence of non-zero scalars $\seq{\lm_k}_{k\in\N}$ such that $\lm_kT^{n_k}x\to 0$ and $\lm_k^{-1}x_k\to 0$. Suppose that $\left\|T^{n_k}x\right\|$ and $\left\|x_k\right\|$ are not both zero. Then, we have the following cases: 
\begin{enumerate}[(1)]
\item if $\left\|T^{n_k}x\right\|\left\|x_k\right\|\neq 0$, set $\lm_k=\left\| x_k\right\|^{\frac{1}{2}}{\left\|T^{n_k}x \right\|}^{-\frac{1}{2}}$,
\item if $\left\|x_k\right\|= 0$, set $\lm_k=2^{-k}\norm{T^{n_k}x}^{-1}$,
\item if $\left\|T^{n_k}x\right\|=0$, set $\lm_k=2^k\norm{x_k}$. 
\end{enumerate}
Indeed, for the last case when  $\left\|T^{n_k}x\right\|=0$, $T$ turns to be $\m$-hypercyclic \cite[Theorem 3.6]{madore2011subspace} and thus $\m$-diskcyclic. Also, for the first two cases if $\left\|T^{n_k}x \right\|\to 0$, then $T$ is $\m$-hypercyclic. Otherwise, it follows easily that $|\lm_k|\le 1$ for all $k\in \N$.
Set $z=x+{\lm_k}^{-1}x_k$ for a large enough $k$. Since $x\in U_1 \subset \m$ and ${\lm_k}^{-1}x_k \in \m$, then $z \in \m$. Since 
$$\norm{z-x}\to 0,$$
it follows that $z\in U_1$.\\
Now, since $\lm_k T^{n_k}z=\lm_k T^{n_k}x+T^{n_k}x_k$, then by using the condition \ref{c104}  both $\lm_k T^{n_k}z$ and $T^{n_k}x_k$ belong to $\m$ and so $\lm_k T^{n_k}x \in \m$.  Moreover, since  $T^{n_k}x_k \to y$  for a large enough $k$, then 
$$\norm{\lm_k T^{n_k}z-y}\to 0.$$
Thus $\lm_k T^{n_k}z \in U_2$. It follows that there exist $k\in\N$ such that $U_1\cap T^{-n_k}\left(\lm_k^{-1}U_2\right)\neq \phi$. 
By \Cref{20} and \Cref{21}, $T$ is an $\m$-diskcyclic operator.\\
\end{proof}
%----------------------------------------------------------------------------------------------------------------------------------
The following lemma can be proved by the same lines as in the proof of \cite[Lemma 3.1.]{feldman2003hypercyclicity} and \cite[Lemma 3.3.]{feldman2003hypercyclicity} respectively.

\begin{lem}\label{35}
Let $T$ be an invertible bilateral weighted shift on $\LZp$ and $\seq{n_k}_{k\in\N}$ be an increasing sequence of positive integers. Suppose that $\m$ is a subspace of $\LZp$ with the canonical basis $\set{e_{m_i}:i\in\N, {m_i}\in \Z}$ such that $T^{n_k}\m \subseteq \m$. If there exists an $i,j\in \N$ such that $T^{n_k}e_{m_i}\to 0$ ($\norm{T^{n_k}e_{m_i}}\norm{B^{n_k}e_{m_j}}\to 0$) as $k\to \iy$, then $T^{n_k}e_{m_r}\to 0$ (or $\norm{T^{n_k}e_{m_r}}\norm{B^{n_k}e_{m_p}}\to 0$, respectively) for all $r,p\in \N$  
\end{lem}
\begin{proof}
Since $T^{n_k}\m \subseteq \m$, the proof is similar to the proof of \cite[Lemma 3.1.]{feldman2003hypercyclicity} and \cite[Lemma 3.3.]{feldman2003hypercyclicity}.
\end{proof}
%----------------------------------------------------------------------------------------------------------------------------------
Now, the next example shows that $\m$-diskcyclicity does not imply to $\m$-hypercyclicity.
\begin{ex}
Let $F:\LZp \rightarrow \LZp$ be a bilateral weighted forward shift operator, defined by $F(e_n) =w_ne_{n+1}$ for all $n\in \Z$, where
\begin{equation*}
w_n= 
\begin{cases} 3 & \text{if $n \geq 0$,}
\\
4 &\text{if $n< 0$.}
\end{cases}
\end{equation*}
Let $\m$ be the subspace of $\LZp$ defined as follows:
 $$\m=\left\{\seq{a_n}_{n=-\iy}^{\iy}\in \LZp : a_{2n}=0, n \in \Z\right\},$$
then $F$ is an $\m$-diskcyclic operator, not $\m$-hypercyclic.
\end{ex}
\begin{proof}
We will apply $\m$-diskcyclic criterion to give the proof. Let $D=D_1 = D_2$ be dense subsets of $\m$, consisting of all sequences with finite support. Let $n_k= 2k$ for all $k\in \N$. 
It is clear that the set $C=\set{e_{m}:m\in O}$ is the canonical basis for $\m$, where $O$ is the set of all odd integer numbers. Let $x,y\in D$, then $x=\sum_{i\in O} x_{i} e_i$ and $y=\sum_{i\in O} y_{i} e_i$, where $x_i,y_i\in\co$ for all $i\in O$.

Let $B$ be a bilateral weighted backward shift on $\LZp$ defined by  $Be_n=z_ne_{n-1}, n\in \Z$, where
\begin{equation*}
z_n= 
\begin{cases} \frac{1}{3} & \text{if $n > 0$,}
\\
\frac{1}{4} &\text{if $n \le 0$.}
\end{cases}
\end{equation*}

Suppose that $x_k=B^{2k}y$ for all $k\in\N$. 
Since $|w_n|\ge 4$ and $|z_n|\ge 1/4$ for all $n\in\Z$, then $F$ and $B$ are invertible with $F^{-1}=B$. Since $B$ and $T$ are linear and invertible, then it is sufficient by triangle inequality and \Cref{35}  to assume that $x=y = e_1$.
Since
 $$\displaystyle B^{2k}e_1=\brc{\prod_{j=0}^{1-2k}z_j}e_{1-2k},$$
it is clear that $\norm{B^{2k}e_1}=\frac{1}{4^{2k}} \to 0$ as $k\to \iy$. Hence
\begin{equation}\label{106}
\norm{x_k} \to 0 \mbox{ as } k\to \iy.
\end{equation}
It is easy to show that for a large enough $k$, 
\begin{equation}\label{105}
F^{2k}x_k=y. 
\end{equation}
It follows from \Cref{106} and \Cref{105} that the condition \ref{a104} in \Cref{sdc} holds. \\
Moreover, we have
$$\displaystyle\norm{F^{2k}e_1} \norm{B^{2k}e_1}=\norm{\prod_{j=1}^{2k}w_j}\norm{\prod_{j=0}^{1-2k}z_j}=\left(\frac{3}{4}\right)^{2k}\to 0,$$
as $k\to \iy$.  Hence the condition \ref{b104} in \Cref{sdc} holds.
It can be easily deduced from the definition of $\m$ that  for each $x \in \m$ and each $k \in \N$, the sequence $F^{2k}x$ will have a zero entry on all even positions, that is 
$$F^{2k}x \in \m.$$
It follows that the condition \ref{c104} in \Cref{sdc} holds.
Thus $F$ is an $\m$-diskcyclic operator.\\
Note that the operator $F$ is clearly not $\m$-hypercyclic since
 $$\displaystyle\norm{F^{n_k}e_i}=\norm{\prod_{j=i}^{i+n_k-1}w_j}\to \iy.$$ 
for any increasing sequence $\seq{n_k}_{k\in\N}$ of positive integers, and any $i\in \Z$, that is, its orbit can not be dense in any subspace.
\end{proof}
The next simple example shows that $\m$-supercyclicity does not imply to $\m$-diskcyclicity.
\begin{ex}
Let $I$ be the identity operator on the space $\co^k$ for some $k\ge 2$, and let $\m$ be a subspace of $\co^k$. Then it is clear that $\co Orb(I,x)\cap \m$ is dense in $\m$ for some vector $0\neq x\in \co^k$, that is, $I$ is $\m$-supercyclic. However, $\D Orb(I,x)\cap \m$ can not be dense in $\m$ for any $x\in \co^k$, that is, $I$ is not $\m$-diskcyclic. 
\end{ex}

%----------------------------------------------------------------------------------------------------------------------------------
The following example gives several useful consequences, some of them answering the corresponding questions to \cite[Question 3.3]{xian2012subspace}, \cite[Question 1]{madore2011subspace} and \cite[Question 1]{rezaei2013notes}, but for subspace-diskcyclicity.
\begin{ex}\label{12}
Let  $T=kx \in \B(\co^n)$, $k\in \D^c, n\ge 2$. Let $\m=\{y:y=(a,0,0,\cdots,0), y\in \co^n\}$ be a subspace of $\co^n$. Then
\begin{enumerate}
\item $T$ and $T^*$ are $\m$-diskcyclic operators, \label{a50}
\item $T^{-1}$ is not subspace-diskcyclic operator for any subspace,\label{b50}
\item There is some vector $x\in \co^n$ such that $\D Orb(T^{-1},x)$ is somewhere dense in $\m$, but not everywhere dense in $\m$. \label{c50}
\end{enumerate}
\end{ex}
\begin{proof}
For (\ref{a50}), let $x=(1,0,0,\cdots,0)$, then 
$$\D Orb(T,x)\cap \m=\set{\brc{\al k^n,0,0,\cdots,0}:\al\in \D, n\ge 0}.$$
Let $z=(b,0,0,\cdots,0)\in \m$, and let us choose an $m\in \N$ such that $|k^m|\ge |b|$. Then it is clear that $z=\left(k^m\brc{\frac{b}{k^m}},0,0,\cdots,0\right)\in \D Orb(T,x)\cap \m$. It follows that $T$ is an $\m$-diskcyclic operator. By the same way, we can show that $T^*=\bar{k}x$ is $\m$-diskcyclic.\\

For (\ref{b50}), since $T^{-1}x={\frac{1}{2}}x$ then $\D Orb(T^{-1},x)$ is bounded for all $x\in \co^n$, and hence $T^{-1}$ can not be dense in any proper subspace of $\co^n$. Thus, $T^{-1}$ is not $\m$-diskcyclic.\\

For (\ref{c50}), let $x=(1,0,0,\cdots,0)$, then $Int\brc{\overline{\D Orb(T^{-1},x)\cap \m}}=\{(y,0,0,\cdots,0):y\in \co,|y|<1\}\neq \phi$. Therfore, $\D Orb(T^{-1}, x)$ is somewhere dense in $\m$. However, by part (\ref{b50}) $\D Orb(T^{-1},x)$ is not everywhere dense in $\m$.

\end{proof}
It follows from the above example that compact and hyponormal subspace-diskcyclic operators exist on $\co$.
%----------------------------------------------------------------------------------------------------------------------------------
Since every two $n$-dimensional Hilbert spaces over the scalar complex field are isomorphic. Then from \Cref{12} above, one may easily conclude the following proposition.
\begin{prop}\label{22} 
There are subspace-diskcyclic operators on every finite dimensional Hilbert space over the scalar field $\co$,
\end{prop}

\end{document}